\documentclass{article}
\usepackage[utf8]{inputenc}
\usepackage{todonotes}
\usepackage{caption}
\usepackage{nicefrac}
\usepackage{amsmath,enumerate}
 \usepackage{amsthm}
\usepackage{amsfonts}
\usepackage{amssymb}
\usepackage{mathrsfs}
\usepackage{pgfplots}
\usepackage{mathtools}
\usepackage{comment}
\usepackage{lmodern}
\usepackage{bm}
\usepackage{subfig}
\usepackage{caption}
\usepackage{tikz}
\usetikzlibrary{shapes,arrows,graphs,graphs.standard,shapes.misc}
\usetikzlibrary{matrix,decorations.pathreplacing, calc, positioning,fit}
\usetikzlibrary{shapes.geometric}

\newtheorem{Theorem}{Theorem}
\newtheorem{Definition}{Definition}

\newtheorem{Proposition}[Theorem]{Proposition}
\newtheorem{Lemma}{Lemma}
\newtheorem{Corollary}[Theorem]{Corollary}

\newcommand{\R}{\mathbb{R}}
\newcommand{\Z}{\mathbb{Z}}

\newcommand{\bu}{{\bm u}}

\newcommand{\bw}{{\bm w}}

\title{On Integer Balancing of Digraphs}
\author{Mohamed-Ali Belabbas\footnote{Coordinated Science Laboratory, University of Illinois, Urbana-Champaign. Email: \texttt{belabbas@illinois.edu}}\,\, and Xudong Chen\footnote{Department of Electrical, Computer, and Energy Engineering, University of Colorado, Boulder. Email: \texttt{xudong.chen@colorado.edu}}}

\begin{document}

\date{}
\maketitle

\begin{abstract}
A weighted digraph is balanced if the sums of the weights of the incoming and of the outgoing edges are equal at each vertex. We show that if these sums are integers, then the edge weights can be integers as well. 
\end{abstract}

\section{Introduction}

Let $G = (V, E)$ be a strongly connected digraph on $n$ vertices, with vertex set $V$ and  edge set $E$. We use $v_iv_j$ to denote an edge from $v_i$ to $v_j$. The digraph $G$ can have self-arcs. For a vertex $v_i$, let $N^-(v_i):=\{v_j \in V \mid v_iv_j\in E\}$ and $N^+(v_i):=\{v_k \in V\mid v_kv_i\in E\}$ be the sets of out-neighbors and in-neighbors of $v_i$, respectively. 

We let $\R_+$ (resp. $\Z_+$) be the set of nonnegative real numbers (resp. nonnegative integers). 
We assign $w_{ij}\in \R_+$ to edges $v_iv_j$, for $v_iv_j\in E$, and denote by $\bw\in \R^{|E|}_+$ the collection of these $w_{ij}$. We call $(G,\bw)$ a weighted digraph.  

\begin{Definition}\label{def:label}
The weighted digraph $(G,\bw)$ is said to be {\em balanced} if, for each vertex, the inflow equals to the outflow:
\begin{equation}\label{eq:balance}
u_i:=\sum_{v_j \in N^-(v_i)} w_{ij} = \sum_{v_k\in N^+(v_i)} w_{ki}, \quad \forall v_i\in V. 
\end{equation}
We call $u_i$ the {\em weight} of vertex $v_i$.
\end{Definition}

The vector $\bu:=(u_1,\ldots,u_n)\in \R^n_+$ is said to be {\em feasible} 
if there exists $\bw\in \R^{|E|}_+$ such that~\eqref{eq:balance} holds.

Balanced digraphs have a host of applications in engineering and applied sciences, including the study of flocking behaviors~\cite{jadbabaie2003coordination},  sensor networks and distributed estimation~\cite{carli2008distributed}. While balancing  over the real numbers is acceptable in some scenarios,  others such as traffic management and fractional packing,  require integer balancing~\cite{garg2007faster,plotkin1995fast,bertsekas1998network, hooi1970class,rikos2017distributed}.

If all the $w_{ij}$ are integers, then clearly every $u_i$ is an integer. 
The question we are interested in is: given a feasible integer-valued $\bu$,  can we find an integer-valued $\bw$? We show that the answer is affirmative: 

\begin{Theorem}\label{th:main1}
Let $G$ be a strongly connected digraph and $\bu \in \Z^n$ be any feasible vector. Then, there exist nonnegative integers $w_{ij}$ such that~\eqref{eq:balance} holds.
\end{Theorem}

The result also applies to the case of weakly connected digraph $G$, based upon the fact that $(G,\bw)$ is balanced if and only if every strongly connected component of $(G, \bw)$ is balanced~\cite{hooi1970class}.

We provide below a constructive proof of Theorem~\ref{th:main1}. 
\section{Algorithm, Propositions, and Proofs}

To proceed, we associate to the digraph $G = (V, E)$ an undirected bipartite graph $B=(X,Y,F)$ on $2n$ vertices, where $X\sqcup Y$ is the vertex set and $F$ is the edge set. Each of the two sets $X$ and $Y$ comprises $n$ vertices.  
The edge set $F$ is defined as follows: there is an edge $(x_i, y_j)$ in $B$ if $v_iv_j$ is an edge of $G$. See Fig.~\ref{fig:onlyfig} for an illustration.

\begin{figure}
    \centering
\subfloat[\label{sfig1:g}]{
\includegraphics{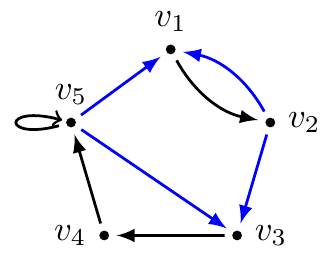}
}\qquad \qquad
\subfloat[\label{sfig2:g}]{
\includegraphics{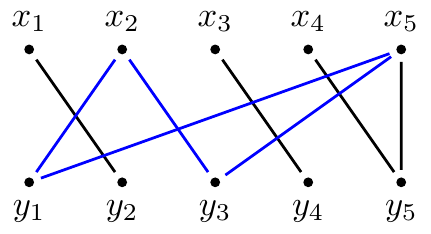}
}
    
    \caption{{\em Left}: A digraph $G$. {\em Right}: Its bipartite counterpart $B$. A cycle in $B$, and the corresponding edges in $G$, are marked in blue.}
    \label{fig:onlyfig}
\end{figure}

Note that the directed edges in $G$ are in one-to-one correspondence with the undirected edges in $B$.  Thus, we can assign the edge weights $w_{ij}$, for $v_iv_j\in E$, to the edges $(x_i, y_j)$ in $B$. 

The balance relation~\eqref{eq:balance}, when applied to the bipartite representation of $G$,  is now turned into 
\begin{equation}\label{eq:balanceBipart}
u_i = \sum_{y_j\in N(x_i)} w_{ij} = \sum_{x_k\in N(y_i)} w_{ki}, \quad \forall i = 1,\ldots, n.
\end{equation}
If the above relations hold for some nonnegative real numbers $u_i$, then $(B,\bw)$ is said to be balanced with vertex weights $u_i$ for both $x_i$ and $y_i$. The following result is  an immediate consequence of the above construction of $(B,\bw)$: 

\begin{Lemma}
The digraph $(G,\bw)$ is balanced if and only if $(B,\bw)$ is balanced.
\end{Lemma}

Now, let $\bu\in \Z^n$ be a feasible vector and $\bw\in \R^{|E|}_+$ be such that~\eqref{eq:balanceBipart} is satisfied. In the sequel, we refer  to elements of $\R_+\backslash \Z_+$ as {\it decimal} numbers. 
 
Every cycle in $B$ has an {even} number of edges, and the number is at least 4.  
A cycle in $B$ does not correspond to a (directed) cycle in $G$, as illustrated in Fig.~\ref{fig:onlyfig}. 
Instead, if $x_{\alpha_1}y_{\beta_1}\cdots x_{\alpha_p}y_{\beta_p}x_{\alpha_1}$ is a cycle in $B$, 
then each vertex $v_{\alpha_i}$ in $G$ has two {\em outgoing} edges $v_{\alpha_i}v_{\beta_i}$ and $v_{\alpha_i}v_{\beta_{i-1}}$ (with $\beta_0$ identified with $\beta_p$) while each vertex $v_{\beta_i}$ has two {\em incoming} edges $v_{\alpha_i}v_{\beta_i}$ and $v_{\alpha_{i+1}}v_{\beta_i}$ (with $\alpha_{p+1}$ identified with $\alpha_1$).  

We next introduce the following definition:

\begin{Definition}\label{def:CPcycles}
An edge in $(B,\bw)$ is called {\em decimal} if its weight is a decimal number. A cycle $C$ in $(B,\bw)$ is {\em completely decimal} if all its edges are decimal.
\end{Definition}

Given a balanced $(B,\bw)$, we  aim to obtain a set of integer edge weights $w^*_{ij}\in \mathbb{Z}_+$ that satisfy~\eqref{eq:balance}. We present below an algorithm that does so in a finite number of steps:

\vspace{.2cm}

\noindent{\bf Algorithm 1:}
\begin{enumerate}
    \item If $(B,\bw)$ does not contain a completely decimal cycle, then the algorithm is terminated. Otherwise, select a completely decimal cycle in $B$: $$C=x_{\alpha_1}y_{\beta_1}\cdots x_{\alpha_p}y_{\beta_p}x_{\alpha_1}.$$
    \item For the selected cycle $C$, find an edge whose weight has the smallest decimal part. Without loss of generality, we assume that the edge is $x_{\alpha_1}y_{\beta_1}$ and the decimal part is $\epsilon := w_{\alpha_1\beta_1} - \lfloor w_{\alpha_1\beta_1} \rfloor$.  Update the weights along the cycle as follows: 
    \begin{equation}\label{eq:updaterule}
    \begin{array}{rcl}
    w_{{\alpha_i}{\beta_i}} & \leftarrow & w_{{\alpha_i}{\beta_i}} -  \epsilon \\
    w_{{\beta_i}{\alpha_{i+1}}} & \leftarrow & w_{{\beta_i}{\alpha_{i+1}}} +  \epsilon
    \end{array}
    \quad \mbox{ for } 1 \leq i \leq p,
    \end{equation}
    where we identify $x_{\alpha_{p+1}}$ with $x_{\alpha_1}$. 
    All the other edge weights remain unchanged. 
\end{enumerate}

Note that one can easily obtain a decimal cycle in $(B,\bw)$, if one exists. Let $x_i$ be a vertex incident to a decimal edge. Denote by $\lambda(x_i)$ the number of decimal edges incident to $x_i$. Since the vertex weight $u_i$ of $x_i$ is integer-valued, then clearly $\lambda(x_i) \geq 2$. 
Now fix a decimal edge $(x_i,y_j)$ in $(B,\bw)$. By the above arguments, $\lambda(y_j) \geq 2$. Thus, there exists another decimal edge incident to $y_j$, say $(y_j,x_k)$ and, similarly, $\lambda(x_k)\geq 2$. Iterating this procedure, we will return to some  previously encountered vertex $x_\ell$, since $B$ is finite.  By construction, the vertices obtained in the process yield a completely decimal cycle.

Theorem~\ref{th:main1} is then a direct consequence of the following result:

\begin{Theorem}\label{thm:algorithm}
Let $(B, \bw)$ be a balanced bipartite graph with integer-valued vertex weights $\bu$ satisfying~\eqref{eq:balanceBipart}. Then, Algorithm 1 terminates in a finite number of steps and returns a nonnegative integer-valued solution $\bw^*$ to~\eqref{eq:balanceBipart}, with $\bu^* = \bu$. 
\end{Theorem}

We establish below Theorem~\ref{thm:algorithm} and start with the following proposition: 

\begin{Proposition}\label{prop:1}
Let $(B,\bw)$ be balanced with vertex weights $\bu$ and $C$ be a completely decimal cycle in $(B,\bw)$. Denote by $(B,\bw')$ the bipartite graph obtained after a one-step update on $\bw$ described by~\eqref{eq:updaterule}. Let $\bu'$ be the vertex weights associated with $(B,\bw')$. Then, $(B,\bw')$ is balanced with $\bu' = \bu$.    
\end{Proposition}

\begin{proof}
If a vertex $x_i$ does not belong to $C$, then none of the edges incident to it are updated. Hence, the summation $\sum_{y_j \in N(x_i)}w_{ij}$ is unchanged. The same argument  applies to any vertex $y_j$ that does not belong to $C$.

Next, denote the cycle by $C=x_{\alpha_1}y_{\beta_1}\cdots x_{\alpha_p}y_{\beta_p}x_{\alpha_1}$. Every vertex in $C$ is incident to exactly two consecutive edges in $C$. For any vertex $x_{\alpha_i}$ in the cycle, we have that
\begin{equation}\label{eq:gimmeabreak}
\sum_{\gamma\in N(x_{\alpha_i})} (w'_{\alpha_i \gamma} -  w_{\alpha_i \gamma}) = (w'_{\alpha_i\beta_{i-1}} + w'_{\alpha_{i} \beta_{i}}) - 
(w_{\alpha_i\beta_{i-1}} + w_{{\alpha_{i}} {\beta_{i}}}).
\end{equation}
We identify $\beta_0$ with $\beta_p$ for the case $i = 1$. 
By~\eqref{eq:updaterule}, the two expressions in parentheses on the right hand side of~\eqref{eq:gimmeabreak} are equal, so the difference is $0$.    
The same arguments can be applied to vertices $y_{\beta_i}$. 

Finally, because $\epsilon$ is the smallest decimal part of the weights on the edges along the cycle $C$, the updated edge weights are nonnegative.
\end{proof}

We next have the following proposition:

\begin{Proposition}\label{prop:2}
Let $(B,\bw)$ be a balanced bipartite graph, with integer-valued vertex weights. Then, the following  statements are equivalent:
\begin{enumerate}
    \item There is no decimal edge; 
    \item There is no completely decimal cycle;
    \item The vector $\bw$ is integer-valued.
\end{enumerate}
\end{Proposition}

\begin{proof}

From the discussion after Algorithm 1, we see that 2 implies 1. Furthermore, it is clear that 1 implies 2 and that 2 implies 3. We show below that 2 implies 3.

Assuming 2 holds, let $F'\subseteq F$ be the collection of decimal edges. 
We show that $F'$ is an empty set. Suppose, to the contrary, that $F'$ is nonempty; then, we let $X'\subseteq X$ and $Y'\subseteq Y$ be the collections of vertices incident to the edges in $F'$. Consider the subgraph $B' = (X',Y',F')$ induced by $X'\sqcup Y'$. Because $(B,\bw)$ does not have a completely decimal cycle, $B'$ is acyclic. 
Denote by $B'_1,\ldots, B'_m$ the connected components of $B'$, each of which is a tree. Pick an arbitrary tree $B'_k$ and a leaf of $B'_k$, say $x_i$. On the one hand, there exists one and only one edge $(x_i, y_j)$ in $B'_k$ such that the weight $w_{ij}$ is decimal. By construction, this edge is also the only decimal edge in $B$ incident to $x_i$. On the other hand, since $(B,\bw)$ is balanced, we have that
$$
 w_{ij}  = u_i -  \sum_{y_{j'} \in N(x_i)\backslash\{y_j\}} w_{ij'}. 
$$
The right hand side of the above expression is integer-valued, which is a contradiction. 
\end{proof}

In fact, more can be said about decimal edges and completely decimal cycles: 

\begin{Corollary}
Let $(B,\bw)$ be a balanced bipartite graph, with integer-valued vertex weights. Then, every decimal edge belongs to a completely decimal cycle. 
\end{Corollary}

\begin{proof}
Suppose that $(x_i,y_j)$ is a decimal edge that is not contained in any decimal cycle; then, the weight $w_{ij}$ will not be affected by executing Algorithm 1. On the other hand, when Algorithm 1 is terminated, there is no completely decimal cycle. By  Prop.~\ref{prop:2}, there is no decimal edge, which is a contradiction. 
\end{proof}

Finally, note that every one-step operation of Algorithm 1 reduces the number of completely decimal cycles by at least one. Thus, 
Theorem~\ref{thm:algorithm} follows as an immediate consequence of Propositions~\ref{prop:1} and~\ref{prop:2}.  

\bibliographystyle{plain}
\bibliography{references}

\end{document}